\documentclass{amsart}
\usepackage{amsmath}
\usepackage{amsthm}
\usepackage{amssymb}
\usepackage{enumerate}

\newtheorem{theorem}{Theorem}[section]
\newtheorem{lemma}[theorem]{Lemma}
\newtheorem{proposition}[theorem]{Proposition}
\newtheorem{corollary}[theorem]{Corollary}

\newtheorem{remar}[theorem]{Remark}
\theoremstyle{definition}
\newtheorem{prob}[theorem]{Open Problem}
\newtheorem{example}[theorem]{Example}
\newenvironment{remark}{\begin{remar}\rm}{\end{remar}}

\newcommand{\QED}{{\unskip\nobreak\hfil\penalty50%
\hskip1em\hbox{}\nobreak\hfil $\Box$%
\parfillskip=0pt \finalhyphendemerits=0 \par\medskip\noindent}}
\newcommand{\EXend}{{\unskip\nobreak\hfil\penalty50%
\hskip1em\hbox{}\nobreak\hfil $\diamondsuit$%
\parfillskip=0pt \finalhyphendemerits=0 \par\medskip\noindent}}

\newcommand{\bfind}[1]{\index{#1}{\bf #1}}

\newcommand{\sn}{\par\smallskip\noindent}

\newcommand{\sep}{^{\rm sep}}

\newcommand{\chara}{\mbox{\rm char}\,}

\newcommand{\Gal}{\mbox{\rm Gal}\,}

\newcommand{\modd}{\mbox{\rm mod }\,}
\newcommand{\crf}{\mbox{\rm crf}\,}

\newcommand{\N}{\mathbb N}
\newcommand{\Z}{\mathbb Z}

 \newcommand{\abaddress}{\par\bigskip \small\rm
Institute of Mathematics,
University of Silesia, \par
Bankowa 14,
40-007 Katowice,
Poland \par
e-mail: anna.rzepka@us.edu.pl}

 \newcommand{\psaddress}{\par\bigskip \small\rm
Institut für Algebra, Technische Universit\"at Dresden, \par
Zellescher Weg 12-14 , 01062 Dresden, Germany \par
e-mail:  piotr.szewczyk@tu-dresden.de}
\begin{document}
\title{Defect extensions and a characterization of tame fields }
\author{Anna Rzepka and Piotr Szewczyk}
\date{5.\ 08.\ 2022}
\begin{abstract}\noindent
We study the relation between two important classes of valued fields: tame fields and defectless fields. We show that in the case of valued fields of equal characteristic or  rank one valued fields of mixed characteristic, tame fields are exactly the valued fields for which all algebraic extensions are defectless fields. In general tame fields form a proper subclass of valued fields for which all algebraic extensions are defectless fields. We introduce a wider class of roughly tame fields and show that every algebraic extension of a given valued field is defectless if and only if its henselization is roughly tame. Proving the above results we also present constructions of Galois defect extensions in positive as well as  mixed characteristic. 
\end{abstract}
\maketitle
\let\thefootnote\relax\footnote{The second author was funded by the Deutsche
Forschungsgemeinschaft (DFG) - 404427454.}

\section{Introduction}
In this paper we consider two notions which play an important role in the theory of valued fields, namely tame fields and defectless fields. 
Both notions appear in connection with attempts to solve deep open problems in positive characteristic, like local uniformization (the local form of resolution of singularities) and the decidability of the Laurent series field over a finite field and of its perfect hull (cf., e.g.,~\cite{[K3]},~\cite{[K1]}). The common underlying problem for these questions is the description of the structure of valued function fields of positive characteristic. The main difficulty that appears here is the phenomenon of defect extensions. A nontrivial defect shows a ``bad behaviour'' of the extension of a given valuation. Therefore, it is desirable to work with valued fields admitting no defect extensions, called defectless fields. 

Another obstacle which we encounter in positive characteristic is the existence of inseparable extensions. The defect and purely inseparable extensions can be eliminated by taking a narrower class of valued fields, namely tame fields. This class turned out to have  properties which were crucial for the proof of strong results, like local uniformization by alteration (cf.~\cite{[K0]}) or model completeness and decidability relative to the elementary theories of the value group and residue
field of a given tame valued field (see~\cite{[KTF]}). 

The aim of the paper is to state the relation between the notions of defectless and tame fields and to describe in which way the property of being a tame field is stronger than being a defectless field. In Sections 3 and 4 we prove the following result. 
\begin{theorem}		\label{charact_equal}
Assume that $(K,v)$ is a nontrivially valued henselian field. If $(K,v)$ is of mixed characteristic, assume moreover  that the value group of the field is archimedean. Then $(K,v)$ is tame if and only if every algebraic extension of $(K,v)$ is a defectless field. 
\end{theorem}

We also show  that for valued fields of mixed characteristic of higher rank the above theorem is in general not true. We construct an example of a mixed characteristic valued field $(K,v)$ which is not tame but every algebraic extension of  $(K,v)$ is a defectless field  (cf.~Example \ref{counterexample}). 
We therefore introduce the wider class of roughly tame fields (cf.~Section~\ref{sec_mixed}), which coincides with the class of tame fields in the equal characteristic case, and prove the following fact.

\begin{theorem}  \label{thm_henrtame}
Every algebraic extension of $(K,v)$ is a defectless field if and only if the henselization of $(K,v)$ is a roughly tame field. 
\end{theorem}

For the proofs of the above theorems, it is crucial to show that under certain conditions, for a given valued field $(K,v)$ we can construct an algebraic extension $(L|K,v)$ such that $(L,v)$ admits a defect extension. 
Since the defect plays a crucial role in deep open problems in valuation theory, a better understanding of the structure  of such extensions is essential.  Using ramification theory, the study of defect extensions of valued fields of positive residue characteristic $p$  can be reduced to the study of normal defect extensions of degree~$p$. Thus we are particularly interested in the structure and constructions of Galois defect extensions of degree~$p$. While proving the above results we also present constructions of Galois defect extensions of prime degree. This provides examples for the further study of the structure of defect extensions. 

In Section~\ref{sec_comar} we describe relations between some generalizations of the notion of tame fields, like roughly tame fields, roughly deeply ramified fields  and semitame fields.

\section{Preliminaries} \label{preliminaries}
In this paper, we denote by $(K,v)$ a field $K$ with a valuation $v$. We denote its value group  by $vK$ and its residue field by $Kv$. If $L|K$ is a field extension, then by $(L|K,v)$ we denote a valued field extension, where $v$ is a valuation of the field $L$ and $K$ is equipped with the restriction of $v$ to $K$. The {\bf residue characteristic exponent} of $(K,v)$ is a natural number $p$ equal to the residue characteristic $Kv$ if it is positive and $p=1$ if char$Kv=0$. We say that $(K,v)$ is of equal characteristic if char$K=$char$Kv$. Otherwise we say that the valued field is of mixed characteristic.  The valuation ring of $(K,v)$ will be denoted by $\mathcal{O}_v$ or $\mathcal{O}_K$.  We identify equivalent valuations. 

Take a valued field $(K,v)$ of residue characteristic $p>0$. We denote by $(vK)_{vp}$ the smallest convex subgroup of $vK$ that contains $vp$ if $\chara K=0$, and set $(vK)_{vp}=vK$ otherwise.

By $\tilde{K}$ and $K^{sep}$ we denote the algebraic closure and separable algebraic closure, respectively, of a given field $K$.

For a torsion-free abelian group $\Gamma$ we set $\tilde{\Gamma}$ to be the divisible hull of $\Gamma$. Fix a prime number $p$. Then the smallest subgroup of $\tilde{\Gamma}$ which contains $\Gamma$ and is $p$-divisible will be called the {\bf $p$-divisible hull of $\Gamma$}  and denoted by $\frac{1}{p^{\infty}}\Gamma$. The smallest subgroup of $\tilde{\Gamma}$ which contains $\Gamma$ and is $n$-divisible for any $n$ prime to $p$ will be called the  {\bf $p'$-divisible hull of $\Gamma$}.

For the basic facts about valued fields and the details of the facts presented in this section, we refer the reader to \cite{[E],[EP],[K5],[W2],[ZS]}.
\subsection{Defect extensions and defectless  fields}

 Take  a valued field $(K,v)$ and an algebraic extension $L$ of $K$. If $v$ extends in a unique way to a valuation of $L$, then $(L|K,v)$ is called a \bfind{unibranched extension.} 
The valued field $(K,v)$ is called \bfind{henselian} if it satisfies Hensel's Lemma. This is equivalent to the fact that its valuation admits a unique 
extension  to the algebraic closure of this field, in other words, $(\tilde{K}|K,v)$  is a unibranched extension. Note that an algebraic extension of a henselian field is again a henselian field, with respect to the unique extension of the valuation. 

Take a valued field $(K,v)$. Then there is a minimal henselian field extension of $(K,v)$, in the sense that it admits a valuation preserving embedding over $K$ in every henselian extension of $(K, v)$. This extension is unique up to valuation preserving isomorphism over $(K,v)$ (cf.~Theorem 17.11 of~\cite{[E]}). We call this extension the \bfind{henselization} of $(K,v)$ and denote it by $(K,v)^h$, or by $K^h$ if $v$ is fixed.

Take a valued field $(K,v)$ and a finite extension $L$ of $K$. If $v_1,\ldots, v_g$ are the distinct extensions of the valuation $v$ of $K$ to the field $L$, then $L|K$ satisfies the \bfind{fundamental inequality} \mbox{(cf.~Corollary 17.5 of \cite{[E]}):}
\[
[L:K]\> \geq \> \sum_{i=1}^g (v_iL:v_iK)[Lv_i:Kv_i]. 
\]
If $g=1$, that is, if $(L|K,v)$ is unibranched, then by the \bfind{Lemma of Ostrowski} (see \cite{[ZS]}, Chapter VI, \S 12, Corollary to Theorem~25):  
\begin{equation} \label{Ostrowski}
[L:K] \> = \>p^n (vL:vK)[Lv:Kv],
\end{equation}
where  $n$ is a nonnegative integer and  $p$ is the residue characteristic exponent of $(K,v)$.
The factor $p^{n}$ is called the {\bf defect} of the extension $(L|K,v)$ and is denoted by $d(L|K,v)$. If it is nontrivial, that is, if $d(L|K,v)>1$, then $(L|K,v)$ is called {\bf a defect extension}. Otherwise $(L|K,v)$ is called a 
{\bf defectless extension}. Note that the second case holds always for valued fields of residue characteristic zero. 
We call $(K,v)$ a  {\bf defectless field} if equality holds in the fundamental inequality for every finite extension of $(K,v)$. Note that if $(K,v)$ is henselian, then it is a defectless field if and only if every finite extension of $(K,v)$ is defectless. Moreover, we have the following result (for the proof see~\cite{[K5]}; a partial proof was given in Theorem 18.2 of~\cite{[E]}).

\begin{theorem}
\label{iffdef}
A valued field $(K, v)$ is defectless if and only if its henselization $(K, v)^h$
is defectless.
\end{theorem}
For the proof see~\cite{[K5]}; a partial proof was given in Theorem 18.2 of~\cite{[E]}. 
 
Take a finite extension  $(L|K,v)$ of henselian fields and a subextension $(E|K,v)$ of $(L|K,v)$. Since degree of a field extension, ramification index and inertia degree are multiplicative, the defect is also  multiplicative: 
\[ 
d(L|K,v)\>=\> d(L|E,v)\cdot d(E|K,v).
\]
This implies that a finite extension of a defectless field is again a defectless field. We also have that the henselization of defectless field is defectless. However, in general an infinite algebraic extension of a defectless valued  field may not be anymore a defectless field (cf. Remark~\ref{remark}). 

Assume that $(K,v)$ is a henselian valued field of positive residue characteristic $p$ and $L|K$ is an extension of degree $p$. Denote again by $v$ the unique extension of the valuation of $K$ to $L$. Note that by equality~(\ref{Ostrowski}), the extension $(L|K,v)$ has nontrivial defect if and only if the corresponding value group and residue field extensions of $(L|K,v)$  are trivial, that is, $(L|K,v)$  is {\bf immediate}. 
Recall that if the field $K$ is itself of positive characteristic $p$, then every Galois extension of degree $p$ is an  Artin-Schreier extension, that is, an extension generated by a root  of a polynomial $X^p-X-c$ for some $c\in K$. 
Assume  now that char$K=0$. In this case assume additionally that $K$ contains a primitive $p$-th root of unity $\zeta_p$. Then every Galois extension of degree $p$ is a  Kummer extension, that is, an extension generated by a root  of a  polynomial $X^p-a$ for some $a\in K$.

\subsection{Tame fields}
An algebraic extension $(L|K,v)$ of henselian valued fields is called \bfind{tame} if every finite subextension $E|K$ of the extension $L|K$ satisfies the following three conditions:
\sn 
{\bf (T1)} if $p=$char$Kv>0$, then  the ramification index $(vE:vK)$ is prime to $p$,
\sn
{\bf (T2)} the residue field extension $Ev|Kv$ is separable,
\sn
{\bf (T3)} $(E|K,v)$ is a defectless extension.
\sn

As in the case of defectless extensions, every algebraic extension of a henselian valued field of residue characteristic zero is tame. 

A henselian valued field $(K,v)$ is said to be \bfind{tame} if the algebraic closure $\tilde{K}$ of $K$ together with the unique extension of $v$ is a tame extension of $(K,v)$.  If $p$ is the residue characteristic of $K$, then directly from the above definition of a tame extension it follows that $(K,v)$ is tame if and only the conjunction of the following  conditions hold
\sn 
{\bf (TF1)} if $p>0$ then the value group $vK$ is $p$-divisible,
\sn
{\bf (TF2)} the residue field $Kv$ is perfect,
\sn
{\bf (TF3)} $(K,v)$ is a defectless field.

Take a valued field $(K,v)$ and fix an extension of $v$ from $K$ to its algebraic closure
$\tilde{K}$. The {\bf absolute ramification field of $(K,v)$} is the ramification field of the normal extension 
$(K\sep|K,v)$. More precisely, it is the fixed field of the closed subgroup 
\[
G^r:=\{\sigma \in \Gal(K^{sep}|K) \mid v(\sigma a-a)>va \textrm{ for all }
 a\in \mathcal{O}_{K^{sep}}\setminus\{0\}\}  
\]
of Gal$(K^{sep}|K)$ (cf.~Corollary 20.6 of \cite{[E]}). The absolute ramification field is denoted by   $(K,v)^r$,or by $K^r$ if $v$ is fixed. 
 If $(K,v)$ is henselian, then the absolute ramification field $K^r$ is a Galois extension of $K$. It is also the unique maximal tame extension of $(K,v)$ (see Theorem 20.10 of~\cite{[E]} and Proposition 4.1 of~\cite{[KPR]}). This shows that every tame extension of valued fields is separable algebraic. Moreover,  $(K,v)$ is a tame field if and only if $K^r=\widetilde{K}$. We obtain also the following property (cf. Lemma 2.13 of \cite{[KTF]}).
 
\begin{lemma}	\label{subext_tame}
Take an algebraic extension $(L|K,v)$ of henselian valued fields and an intermediate field $E$. Then the extension $(L|K,v)$ is tame if and only if $(L|E,v)$ and $(E|K,v)$ are tame.
\end{lemma}

\begin{theorem}		\label{tame_valgp_res}
Take a valued field $(K,v)$ of residue characteristic exponent $p$ and fix an extension of $v$ to $K^{sep}$. 
Then $vK^r$ is the $p'$-divisible hull of $vK$ and $K^rv$ is the separable algebraic closure of $Kv$. 

Moreover, every finite extension of $K^r$ is a tower of normal extensions of degree~$p$.
\end{theorem}
For the proof of the above theorem see equation~(2.2) and Lemma~2.9 of~\cite{[Ku6]} (cf. also Theorem 7.27 of \cite{[K5]}). A direct consequence of the above fact is that  
$v\tilde{K}$ is the $p$-divisible hull of $vK^r$ and $\tilde{K}v$ is the perfect hull of $vK^r$. 

\begin{lemma}	\label{ram_field_ext}
Assume that $(L|K,v)$ is an algebraic extension of henselian fields. Then $L^r=L.K^r$ with respect to the unique extension of $v$ to the algebraic closure of the fields. 
\end{lemma}
For the proof see Lemma~2.7 of~\cite{[Ku6]}.

A well known property of tame fields is the following fact.
\begin{proposition} \label{tame_impl_def}
If $(K,v)$ is a tame field, then every algebraic extension of $(K,v)$ is a tame and consequently a defectless field. 
\end{proposition}
\begin{proof}
Note first that since $(K,v)$ is henselian, the valuation extends in a unique way to any algebraic extension of $K$. Denote this extension again by $v$.

Assume that $(K,v)$ is a tame field.  Then $K^r=\tilde{K}$. Take an algebraic extension $L$ of $K$. Then  Lemma~\ref{ram_field_ext} yields that
\[ L^r=K^r.L=\tilde{K}.L=\tilde{L}.\]
Hence $(L,v)$ is a tame field and consequently it is defectless. 
\end{proof}

Another strong property of tame extensions is that the defect is invariant under lifting up through such extensions. For the proof of the following fact see Proposition~2.8 of~\cite{[Ku6]}.  
\begin{proposition} \label{defecless_ramfield}
Take a henselian field $(K,v)$ and a tame extension $(N|K,v)$. Assume that $(L,v)$ is a finite extension of $(K,v)$. Then $d(L|K,v)=d(L.N|N,v)$. 

Moreover, the valued field $(K,v)$ is defectless if and only if $(N,v)$ is defectless.  In particular, $(K,v)$ is defectless if and only if its absolute ramifiaction field $(K^r,v)$ is a defectless field. 
\end{proposition}

An  extension $(L,v)$ of a henselian valued field $(K,v)$ is called \bfind{purely wild} if it satisfies  the following  conditions:
\sn 
{\bf (P1)} $vL/vK$ is a $p$-group, where $p$ is the residue characteristic exponent of $(K,v)$
\sn
{\bf (P2)} the residue field extension $Lv|Kv$ is purely inseparable.
\sn
In particular every purely inseparable extension is purely wild.
The following characterization of algebraic purely wild extensions was proved in~\cite{[KPR]}, Lemma~4.2.
\begin{theorem}  
An algebraic extension $(L|K,v)$ of henselian valued fields is purely wild if and only if $L|K$  is linearly disjoint from $K^r|K$.
\end{theorem}
We will also need the following easy observation.
\begin{lemma} 	\label{purelywild_subext}
Take a  subextension $(E|K,v)$ of a valued field extension $(L|K,v)$. Then  $(L|K,v)$ is purely wild if and only if $(E|K,v)$ and $(L|E,v)$ are purely wild. 
\end{lemma}

\subsection{Compositions of valuations}
Take a valued field $(K,v)$ 
and a valuation $\bar{v}$ of $Kv$. Then by $v\circ\bar{v}$ we denote the valuation of $K$ whose valuation ring is of the form
\[
\mathcal{O}_{v\circ\bar{v}}:=\{a\in \mathcal{O}_v : av\in \mathcal{O}_{\bar{v}}\}.
\]
We call $v\circ\bar{v}$ the {\bf composition of the valuations} $v$ and $\bar{v}$. Note that $v$ is a coarsening of $v\circ\bar{v}$. Write $w=v\circ\bar{v}$. The value group $\bar{v}(Kv)$ of $(Kv,\bar{v})$ is isomorphic to a convex subgroup $H$ of $wK$, and $vK\cong wK/H$ (cf. \cite{[EP]}, Section 2.3). 

\begin{lemma}  \label{coars_hens}
Take a valued field $(K,v)$ and a valuation $\bar{v}$ of $Kv$. Then $(K, v\circ\bar{v})$ is a henselian field if and only if $(K,v)$ and $(Kv,\bar{v})$ are henselian fields. 
\end{lemma}
For the proof of the above lemma see Corollary~4.1.4 of~\cite{[EP]}. The next property follows directly from Lemma~4.9 of \cite{[AJ]}. 
\begin{lemma}   \label{coars_def}
Take a henselian field $(K,v)$ and a henselian valuation $\bar{v}$ of $Kv$. Then $(K, v\circ\bar{v})$ is a defectless field if and only if $(K,v)$ and $(Kv,\bar{v})$ are defectless fields. 
\end{lemma}

\section{tame fields and defect extensions in positive characteristic}
We start this section with a proof of Theorem~\ref{charact_equal} for valued fields of equal characteristic. In the case of defectless fields of positive characteristic which are not tame we then give examples of constructions of algebraic extensions of these fields admitting defect extensions. This provides an alternative proof of the theorem.

{\bf Proof of Theorem~\ref{charact_equal}}[the equicharacteristic case].\\
Since $(K,v)$ is henselian, $v$ extends  in a unique way to $\tilde{K}$. Denote this extension again by $v$.

If $(K,v)$ is a tame field, then by Proposition~\ref{tame_impl_def} every algebraic extension of $(K,v)$ is a defectless field.

Note first that if char$Kv=0$, then every henselian field is tame, so it is enough to prove the theorem in the case of char$K=$char$Kv=p>0$.

Suppose now that $(K,v)$ is not a tame field. We show that then there is an algebraic extension of $(K,v)$, which is not a defectless field. Since $(K,v)$ is not tame, condition (TF1), (TF2) or (TF3) does not hold. If $(K,v)$ is not defectless, then the assertion is trivial. Hence we may assume that the value group of $(K,v)$ is not $p$-divisible or its residue field is not perfect.  
Take $K^r$ to be the absolute ramification field of $(K,v)$. If $vK$ is not $p$-divisible, then by Theorem~\ref{tame_valgp_res} also $vK^r$ is not $p$-divisible. Similarly, if $Kv$ is not perfect, then by Theorem~\ref{tame_valgp_res} also $K^rv$ is imperfect. 
Hence $K^r$ is not separable-algebraically closed. Take $L$ to be the perfect hull of $K^r$. Then $vL$ is the $p$-divisible hull of $vK^r$ and $Lv$ is the perfect hull of $K^rv$. Together with Theorem~\ref{tame_valgp_res} this implies that $vL=\widetilde{vK}$ and $Lv=\widetilde{Kv}$. Moreover, since $K^r$ is not separable-algebraically closed, the same is true for $L$. Hence there is a nontrivial finite extension of $(E|L,v)$. As $vL$ is divisible and $Lv$ is algebraically closed, $(vE:vL)=[Ev:Lv]=1$ and thus $(E|L,v)$ is a defect extension.
\QED

{\bf Throughout the remaining part of this section we assume that $(K,v)$ is a nontrivially valued field of positive characteristic $p$.}

Assume that $(K,v)$ is a henselian field whose value group is not $p$-divisible or whose  residue field is not perfect. In general, it is not necessary to pass to the perfect hull of the absolute ramification field of $(K,v)$ to obtain a defect extension. The following examples give precise constructions of such extensions, cf. also Example~3.9 of~\cite{[Ku1]}.

We start with the case of non-p-divisible value group. The following example is a generalization of Example~4.16  of~\cite{[Ku6]}. 
\begin{example}  \label{valgp_eqal}
Assume that $vK$ is not $p$-divisible. Fix an extension of $v$ to a valuation of $\tilde{K}$ and denote it again by $v$. Take an element $a_0\in K$  of negative value such that $\frac{1}{p}va_0\notin vK$. 

By induction on $i$ take $a_i\in \tilde{K}$ such that $a_i^p=a_{i-1}$. Since each element $a_i$, $i\in\N$, is purely inseparable over $K$, the valuation $v$ extends in a unique way to $F:=K(a_i | i\in\N)$. Denote this extension again by $v$. Note that $va_n=\frac{1}{p^n}va_0$. As $\frac{1}{p}va_0\notin vK$, by the fundamental inequality we obtain that
\[
p^n\leq (vK(a_n):vK)[K(a_n)v:Kv]\leq [K(a_n):K]\leq p^n
\]
for every $n\in\N$.
Hence equality holds everywhere. Thus in particular $(K(a_n)|K,v)$ is a unibranched extension. Moreover, $(vK(a_n):vK)=p^n$ and $[K(a_n)v:Kv]=1$, which implies that $vK(a_n)=vK+\frac{1}{p^n}va_0\Z$ and $K(a_n)v=K$. 
Since $F=\displaystyle\bigcup_{n\in\N}K(a_n)$, we deduce that $(F|K,v)$ is a unibranched extension with $vF=vK+\frac{1}{p^{\infty}}(\Z va_0)$ and $Fv=Kv$. 

Take a root $\vartheta$ of a polynomial 
\[
X^p-X-a_0.
\]
Since $va_0<0$, we have that $v\vartheta =\frac{1}{p}va_0<0$. We consider now the extensions $(K(a_n,\vartheta)|K(a_n))$. For every $n\in\N$ set
\[
b_n:=\vartheta -\sum_{i=1}^n a_i\in K(a_n,\vartheta).
\]
Since $\vartheta^p-\vartheta -a_0=0$, we obtain that
\[
0=b_n^p+\sum_{i=1}^n a_i^p- b_n-\sum_{i=1}^n a_i-a_0=b_n^p-b_n-a_n.
\]
As $va_n<0$, $vb_n=\frac{1}{p}va_n=\frac{1}{p^{n+1}}va_0\notin K(a_n)$.  This together with the fundamental inequality  yields that
\[
p\leq (vK(a_n,\vartheta):vK(a_n))[K(a_n,\vartheta)v:K(a_n)v]\leq [K(a_n,\vartheta):K(a_n)]\leq p.
\]
Hence equality holds everywhere and $(K(a_n,\vartheta)|K(a_n),v)$ is a nontrivial unibranched Artin-Schreier extension with 
$vK(a_n)=vK+\frac{1}{p^{n+1}}va_0\Z$  and $K(a_n,\vartheta)v=K(a_n)v$. 

Since $F(\vartheta)= \bigcup_{n\in\N}K(a_n,\vartheta)$, we obtain that $(F(\vartheta )|F,v)$ is a nontrivial unibranched Artin-Schreier extension with a trivial residue field extension. Moreover, as  $vF=vK+\frac{1}{p^{\infty}}(\Z va_0)$, also $vF(\vartheta)=vF$. Hence it is a defect extension. 
\EXend
\end{example}

We give also a similar construction in the case of $p$-divisible value group and imperfect residue field. We will need the following fact (cf.  Lemma 2.13 of~\cite{[K4]} and the construction given in the proof of the Lemma).
\begin{lemma}	\label{res_ASconstr}
Assume that the residue field of $(K,v)$ is imperfect. For an element  \mbox{$\xi \in (Kv)^{1/p}\setminus Kv$,} take  $c\in \mathcal{O}_K$ with $cv=\xi^p$ and $d\in K$ of negative value. Then a root $\vartheta$ of the polynomial $X^p-X-d^pc$ generates an Artin-Schreier extension of $(K,v)$ with $vK(\vartheta)=vK$ and $K(\vartheta)v=Kv(\xi)$.
\end{lemma}

\begin{example}   \label{resf_eqal}
Assume that $vK$ is $p$-divisible and $Kv$ is not a perfect field. 
Pick some element $a_0\in K$ with $va_0<0$. By induction on $i$ take $a_i\in \tilde{K}$ such that $a_i^p=a_{i-1}$ and set $K_0:=K(a_i | i\in\N)$. Note that if $K$ contains an element $a$ of negative value and such that $a^{1/p^n}\in K$ for all $n\in\N$ we can choose $a_0=a$ and $K_0=K$. 

Since $(K_0|K,v)$ is a purely inseparable (possibly trivial) extension, it is unibranched. We consider now two possibilities: $K_0v$ perfect or $K_0v$ an imperfect field. Note that the latter holds in particular if $K_0=K$. 

Assume first that $K_0v$  is a perfect field. Since $Kv$ is not perfect, there is some $\xi \in (Kv)^{1/p}\setminus Kv$. Take $c\in \mathcal{O}_K$ with $cv=\xi^p$ and $d\in K$ with $vd<0$. Let $\vartheta$ be a root of the polynomial 
$X^p-X-d^pc$. Then by Lemma~\ref{res_ASconstr}, $(K(\vartheta)|K,v)$ is a nontrivial Artin-Schreier extension with $vK(\vartheta)=vK$ and $K(\vartheta)v=Kv(\xi)$. Hence in particular $(K(\vartheta)|K,v)$ is a unibranched extension. Since $K_0|K$ is a purely inseparable extension, $(K_0(\vartheta)|K_0,v)$ remains a nontrivial Artin-Schreier extension. 

Since the extension $K_0(\vartheta)|K(\vartheta)$ is purely inseparable, it is unibranched and purely wild. As $K(\vartheta)v|Kv$ is a purely inseparable extension of degree equal to $[K(\vartheta):K]$ , also $(K(\vartheta)|K,v)$ is a unibranched purely wild extension. By Lemma~\ref{purelywild_subext} this yields that $(K_0(\vartheta)|K,v)$ is a unibranched purely wild extension and consequently the same holds for $(K_0(\vartheta)|K_0,v)$.  Note that $vK_0$ is $p$-divisible, as this holds already for $vK$. Also the field $K_0v$  is perfect by assumption. Together with the fact that  $(K_0(\vartheta)|K_0,v)$ is purely wild, this yields that the corresponding value group and residue field extensions are trivial and thus $(K_0(\vartheta)|K_0,v)$ is an Artin-Schreier defect extension. 

Assume now that the field $K_0v$  is imperfect. Take $\xi \in (Kv)^{1/p}\setminus Kv$ and take $c_0\in \mathcal{O}_K$ with $c_0v=\xi^p$. By induction on $i$ take $c_i\in \tilde{K}$ such that $c_i^p=c_{i-1}$ and set \[F:=K_0(c_i | i\in\N).\] Since for each $n\in\N$ we have that $c_nv=\xi^{1/p^{n-1}}$ and
\[
p^n\leq [K_0(c_n)v:K_0v]\leq [K_0(c_n):K_0]\leq p^n,
\]
similarly to Example~\ref{valgp_eqal} we deduce that $(K_0(c_n)|K_0,v)$ is a unibranched extension of degree $p^n$ with 
$K_0(c_n)v=K_0v(\xi^{1/p^{n-1}})$, $vK_0(c_n)=vK_0=vK$. Thus $vF=vK$ and $Fv=K_0v(\xi^{1/p^n}| n\in\N)$. 

Take a root $\vartheta$ of a polynomial 
\[
X^p-X-a_0c_0.
\]
For every $n\in\N$ set
\[
b_n=\vartheta-\sum_{i=1}^na_ic_i\in K_0(c_n,\vartheta).
\]
Note that
\begin{equation} \label{eq1}
b_n^p-b_n=\vartheta^p-\sum_{i=1}^n(a_ic_i)^p-\vartheta+\sum_{i=1}^na_ic_i =a_nc_n.
\end{equation}
Since $v(a_nc_n)<0$, we have that $vb_n=\frac{1}{p}v(a_nc_n)= \frac{1}{p}va_n>va_n$. Together with equation~(\ref{eq1}) and the choice of $a_n$ this shows that 
\[
c_n=\frac{b_n^p}{a_n}- \frac{b_n}{a_n}= \left(\frac{b_n}{a_{n+1}}\right)^p-\frac{b_n}{a_n}
\]
is a unit with 
\[
c_nv=\frac{b_n^p}{a_n}v= \left(\frac{b_n}{a_{n+1}}v\right)^p.
\]
Consequently, $\frac{b_n}{a_{n+1}}v=\xi^{1/p^{n}}\notin K_0(c_n)v. $ Similarly to the previous examples we deduce that $(K_0(c_n,\vartheta)|K_0(c_n),v)$ is a unibranched Artin-Schreier extension with $vK_0(c_n,\vartheta)=vK$ and $K_0(c_n,\vartheta)v=K_0v(\xi^{1/p^{n}})$. Therefore $F(\vartheta)|F,v)$ is a nontrivial unibranched Artin-Schreier extension which is immediate. Hence it is an Artin-Schreier defect extension. 
\EXend
\end{example}

\begin{remark} \label{remark}
Note that the valued fields $(K,v)$ considered in Examples~\ref{valgp_eqal} and~\ref{resf_eqal} may be chosen to be defectless. Indeed, if $(K',v)$ is a maximal immediate extension of $(K,v)$, then $vK'=vK$, $K'v=Kv$ and the field $(K',v)$ is defectless (by Theorem~31.21 of~\cite{[W2]}). However, they admit infinite algebraic extensions which are not defectless fields. 
\end{remark}
\section{tame fields and defect extensions in mixed characteristic}  \label{sec_mixed}
We wish to consider the analogous relation between tame and defectless fields in the case of mixed characteristic. Note that in this case we are not able to repeat the argumentation of the proof of Theorem~\ref{charact_equal} as the valued fields are of characteristic zero and thus are perfect. 
Also, it is relevant for the properties of a given valued field $(K,v)$ of positive residue characteristic $p$ to consider the $p$-divisibility of $(vK)_{vp}$ in place of the $p$-divisibility of $vK$ (cf. ~\cite{[KR]}, \cite{[HH]}, \cite{[J]}). If $K$  itself is of characteristic $p$ or $vK$ is of rank 1,  both groups $vK$ and $(vK)_{vp}$ are equal. But in the case of valued fields of mixed characteristic and higher rank, $(vK)_{vp}$ may be a proper subgroup of $vK$. The following example shows that in this case the equivalence from Theorem~\ref{charact_equal} may not hold.

\begin{example} \label{counterexample}
Take a tame field $(K,v_1)$ of characteristic 0 and residue characteristic $p>0$. 
 Take  $K(x)$ to be the rational function field over $K$ and take $v_x$  to be the $x$-adic valuation  on $K(x)$ and set $v_xx=1$. Take now $F$ to be the henselization of $K(x)$ with respect to $v_x$. Note that $Fv_x=K(x)v_x=K$, so we may consider the composition $v=v_x\circ v_1$ and the valued field $(F,v)$. 
Note that $vF=\Z\times v_1K$ ordered lexicographically and 
 for a polynomial $f= \displaystyle\sum_{i=1}^n a_ix^i \in K[x]$  the summands have distinct values, so we have
\[
v(f)=\min_{1\leq i\leq n}va_ix^i.
\]
Since $(F,v_x)$ is henselian and $(K,v_1)$  is henselian as a tame field, Lemma \ref{coars_hens} yields that also $(F,v)$ is henselian. 

Take an algebraic extension $(L,v')$ of $(F,v)$. Then $v'=v'_x\circ v'_1$, where $v'_x$ is the unique extension of $v_x$ to $L$ and $v'_1$ is the unique extension of $v_1$ to the algebraic extension $Lv'_x$ of $K$. Since $(K,v_1)$ is tame, by Proposition~\ref{tame_impl_def} the valued field $(Lv'_x,v'_1)$ is defectless. Also the field $(L,v'_x)$ is defectless, since it is a henselian field of residue characteristic zero. Therefore by Lemma~\ref{coars_def} the field $(L,v')$ is defectless. 
Thus every algebraic extension of $(F,v)$ is a defectless field. 
On the other hand, the value group $vF=\Z\times v_1K$   is not $p$-divisible. Hence $(F,v)$ is not a tame field.
\EXend
\end{example} 
The above example shows that $p$-divisibility of the whole value group of $(K,v)$ is not a necessary condition for the valued field to ensure that all algebraic extensions of $(K,v)$ are defectless fields. We therefore relax the condition (T1) and consider a more general class of fields. 

Take a valued field $(K,v)$ and denote by $p$  the residue characteristic exponent of $Kv$. We call the value group $vK$ \bfind{roughly $p$-divisible} if the group $(vK)_{vp}$ is $p$-divisible (cf.~\cite{[J]}).

 \bfind{ In the remaining part of this section for a valued field $(K,v)$ we will consider the composition $v=v_0\circ v_1$, where  $v_0$ is the finest coarsening of $v$ that has residue characteristic 0.}
 We call $(Kv_0,v_1)$ the \bfind{core field of $(K,v)$}.  Note that if $(K,v)$ is henselian, then by Lemma~\ref{coars_hens} both fields $(K,v_0)$
 and $(Kv_0,v_1)$ are henselian. If char$Kv=0$, then $v_1$ is a trivial valuation  and if $\chara K>0$, then $v_0$ is trivial. 
Note also that if $\chara Kv=p>0$, then the value group of the core valued field  $(Kv_0,v_1)$ is equal to $(vK)_{vp}$. 

\begin{proposition}   \label{coars_mixed_prop}
Take a henselian field $(K,v)$. Then  $(K,v)$ is a defectless field if and only the core field of $(K,v)$ is a defectless field. 
Moreover, every algebraic extension of $(K,v)$ is a defectless field if and only if every algebraic extension of the core field of $(K,v)$ is a defectless field. 
\end{proposition} 
\begin{proof} 
Since $(K,v_0)$ is a henselian field of residue characteristic 0, it is a defectless field. Lemma~\ref{coars_def} yields that $(K,v)$ is a defectless field if and only if the core field $(Kv_0,v_1)$ of $(K,v)$ is defectless.

For the proof of the second assertion, assume first that  every algebraic extension of the core field $(Kv_0,v_1)$ of $(K,v)$ is a defectless field. Take an algebraic extension  $(L,v')$ of $(K,v)$. Then we can write $v'=v'_0\circ v'_1$, where $v'_0$ is the unique extension of $v_0$ to $L$ and  $v'_1$ is the unique extension of $v_1$ to $Lv'_0$. Note that $L|K$ and $v'_0|Kv_0$  are algebraic extensions. In particular the ranks of $v'_0$  and $'_1$ remains the same as the rank of $v_0$ and $v_1$ respectively. This yields that  $(Lv'_0,v'_1)$ is the core field of $(L,v')$. Hence   $(Lv'_0,v'_1)$  is defectless as an algebraic extension of $(Kv_0,v_1)$. thus, by the first part of the proof, also $(L,v')$ is a defectless field.

Assume now that every algebraic extension of $(K,v)$ is a defectless field. Take an algebraic extension $F$ of $Kv_0$. As  $(Kv_0,v_1)$ is henselian, there is a unique extension $v'_1$ of $v_1$ to $F$. Since $F|Kv_0$ is algebraic, there is an algebraic extension $L$ of $K$ such that $Lv'_0=F$, where $v'_0$ is the unique extension of $v_0$ to $L$ (cf. Theorem 2.14 of~\cite{[K4]}). Also $v'=v'_0\circ v'_1$ is the unique extension of $v$ to $L$. By assumption $(L,v')$ is a defectless field, so Lemma~\ref{coars_def} yields that also the field $(Lv'_0,v'_1)=(F,v'_1)$ is  defectless.
\end{proof}

We call a henselian valued field $(K,v)$  \bfind{a roughly tame field} if the core field of $(K,v)$  is a tame field. 
A direct consequence of the above definition is the following property.
\begin{lemma} \label{rtame_equchar}
If $(K,v)$ is a valued field with $\chara K=\chara Kv$ or  $\chara Kv>0$ and $(vK)_{vp}=vK$, then the field is tame if and only if it is roughly tame.
\end{lemma}

As in the case of tame fields, the property of being roughly tame is preserved under algebraic extensions. 

\begin{lemma} \label{rt}
Let $(K,v)$ be a roughly tame field. Then any algebraic extension $(L|K,v)$ is a roughly tame field.
\end{lemma}
\begin{proof}
By assumption the core field of $(K,v)$ is a tame field. Since $(L,v)$ is an algebraic extension of $(K,v)$, also the core field of $(L,v)$ is an algebraic extension of the core field of $(K,v)$.  Hence by Proposition~\ref{tame_impl_def} the core field of $(L,v)$ is tame and consequently the valued field $(L,v)$ is roughly tame.
\end{proof}

\begin{proposition} \label{roughly_charact}
Take a henselian valued field $(K,v)$ of  residue characteristic $p$. Then $(K,v)$ is roughly tame if and only if the following conditions hold
\sn 
{\bf (RTF1)} $vK$ is roughly $p$-divisible
\sn
{\bf (RTF2)} the residue field $Kv$ is perfect,
\sn
{\bf (RTF3)} $(K,v)$ is a defectless field.
\end{proposition}
\begin{proof}
Observe that $v_1(Kv_0)=(vK)_{vp}$ and $(Kv_0)v_1=Kv$. Hence $(Kv_0,v_1)$  satisfies (TF1) if and only if (RTF1) holds for $(K,v)$. Similarly, (TF2) holds for  $(Kv_0,v_1)$ if and only if (RTF2) holds for $(K,v)$. It remains to consider the conditions (TF3) for $(Kv_0,v_1)$ and (RTF3) for $(K,v)$. The equivalence of these conditions follows directly from Proposition~\ref{coars_mixed_prop}.
\end{proof}

Note that every tame field satisfies (RTF1)-(RTF3). We thus obtain the following property.
\begin{corollary} 
Every tame field is a roughly tame field. 
\end{corollary}

\sn

We present now a construction of Galois defect extensions in the case of valued fields of mixed characteristic. 
We will need the following observation (for the proof see Lemma 3.19 of~\cite{[Ku1]}).
\begin{lemma}                                                     \label{congruece}
Assume that $(K,v)$ is a valued field of characteristic 0 and residue characteristic $p$. Take elements $c_1,\ldots,c_n\in K$ of values $\geq - \frac{vp}{p}$. Then
\[
(c_1+\cdots + c_n)^p \equiv c_1^p+\cdots + c_n^p \,(\hspace{-3mm}\mod \mathcal{O}_K).
\]
\end{lemma}

\begin{example} \label{ex_valgp_mixed}
We consider now a generalization of Example~3.22 of~\cite{[Ku1]}. Assume that $(K,v)$ is a field of characteristic 0 and residue characteristic $p>0$ such that $(vK)_{vp}$ is not $p$-divisible.
Then there is $a_0\in K$ with $-vp\leq va_0<0$ and $\frac{1}{p}va_0\notin vK$. Set $\alpha:=va_0$. Fix an extension of $v$
to $\tilde{K}$ and denote it again by $v$. 
Set $F_0=K(\zeta_p)$, where $\zeta_p$ denotes a primitive $p$-th root of unity. We obtain that $\frac{1}{p}\alpha\notin vF_0$. 

Take $a_1$ to be a root of the polynomial $X^p-X-a_0$. By induction, choose $a_{i+1}$ to be a root of the
polynomial $X^p-X+a_i$. Further, take $a$ to be a root of the polynomial
$X^p-a_0$. Then $va=\frac{1}{p}va_0\geq -\frac{vp}{p}$. By induction on $i$, we obtain that  $va_i=\frac{1}{p^i}\alpha$. Hence $-\frac{vp}{p}\leq va_i<0$. 

As in Example~\ref{valgp_eqal} we deduce that  $(F_0(a_i)|F_0,v)$ is a unibranched extension with $vF_0(a_i)=\Z\frac{1}{p^i}\alpha+vF_0$ and
$F_0(a_i)v=F_0v$. 
Moreover, for $\displaystyle F:=\bigcup_{i\in\N}F_0(a_i)$ we obtain that $(F|F_0,v)$ is a unibranched extension with 
$vF=\frac{1}{p^{\infty}}\big(\Z\alpha\big)+vF_0$ and $Fv=F_0v$.

For every $k\geq 2$ set
\[
b_k=a-\sum_{i=1}^ka_i\in F_0(a_k,a).
\]
Since for $i\geq 1$ we have $a_{i+1}^p=a_{i+1}-a_i$ and $a_1^p=a_1+a_0$, Lemma~\ref{congruece} yields that 
\begin{eqnarray*}
0 & = & a^p-a_0=\left(b_k+\sum_{i=1}^k a_i\right)^p-a_0\\
& \equiv  & b_k^p+\sum_{i=1}^k a_i^p-a_0 = b_k^p+a_k (\hspace{-3mm}\mod \mathcal{O}_v).
\end{eqnarray*}
Hence $v(b_k^p+a_k)\geq 0$ and thus $vb_k=\frac{1}{p}va_k=\frac{1}{p^{k+1}}\alpha$. Repeating the argumentation from Example~\ref{valgp_eqal} we obtain that the extension $F_0(a_k,a)|F_0(a_k)$ is unibranched of degree $p$ with 
$F_0(a_k,a)v=F_0(a_k)v$ and
$vF_0(a_k,a)=\Z(\frac{1}{p^{k+1}}\alpha)+vF_0$.
Therefore
$F(a)|F$ is a unibranched Galois extension of degree $p$. Moreover, $F(a)v=Fv$ and  $vF(a)= vF$. We thus obtain that $(F(a)|F,v)$ is a Kummer defect extension of degree $p$.
\EXend
\end{example}

The construction presented in Example~\ref{ex_valgp_mixed} proves in particular the following fact.

\begin{lemma} \label{prop_valgp_mixed}
Assume that $(K,v)$ is a valued field of characteristic zero and positive residue characteristic $p$, with $(vK)_{vp}$ not $p$-divisible. Then there is a unibranched algebraic extension $(F|K,v)$ such that the field $(F,v)$ is not defectless.
\end{lemma}

Take a valued field $(K,v)$ of characteristic 0 and positive residue characteristic~$p$. Consider the sequence of elements $d_i\in \tilde{K}$:
\begin{equation} \label{di}
d_0\in K\textrm{ with } -vp\leq vd_0< 0 \textrm{ and } d_i^p=d_{i-1}, \, i\in\N. 
\end{equation}

\begin{example} \label{ex_resf_mixed}
Assume that $(K,v)$ is a valued field of characteristic $0$ and  residue characteristic $p>0$ with  roughly $p$-divisible value group 
 and imperfect residue field $Kv$. Assume additionally that there is a sequence of elements $d_i$ 
as in~(\ref{di}) and an extension of $v$ to $K_0:=K(d_i | i\in\N)$ such that the residue field $K_0v$ is imperfect. 

Replacing $K_0$ by $K_0(\zeta_p)$, if necessary, we assume that $K_0$ contains a primitive $p$-th root of unity.

Take $a_0\in\mathcal{O}_{K_0}$ such that $\eta :=a_0v\notin (K_0v)^p$.  Set $b_0:=a_0d_0$ and take $b_1$ to be a root of the polynomial $X^p-X-b_{0}$. 
By induction, choose $b_{i}$ to be a root of the polynomial $X^p-X+b_{i-1}$ and set $a_i=\frac{b_i}{d_i}$. 

Assume we have shown that $va_i=0$ and $a_iv=\eta^{1/p^i}$ for some $i\geq 0$. Then in particular $vb_i=vd_i<0$. 
Since $v(b_{i+1}^p-b_{i+1})=vb_i<0$, we deduce that $vb_{i+1}<0$ and $vb_{i+1}^p=vb_i$. Thus  $vb_{i+1}>vb_i$ and
\[
v\left( \frac{b_{i+1}^p}{d_i}+a_i\right)=v(b_{i+1}^p+b_i)-vd_i=vb_{i+1}-vd_i>vb_i-vd_i=0.
\]
Hence $\left(\frac{b_{i+1}}{d_{i+1}}\right)^pv=\frac{b_{i+1}^p}{d_i}v=-a_iv$ and consequently,
\[a_{i+1}v=\frac{b_{i+1}}{d_{i+1}}v=-eta^{1/p^{i+1}}.\]
In particular we have that $va_{i+1}=0$. As in the previous examples we deduce that $(K_0(b_i)|K_0,v)$ is a unibranched extension with $vK_0(b_i)=vK_0$ and $K_0(b_i)v=K_0v\left(\eta^{\frac{1}{p^i}}\right)$.

Note that for every $i\geq 0$,
\begin{equation} \label{vbi}
vb_i=vd_i=\frac{1}{p^i}vd_0\geq -\frac{vp}{p^i}.
\end{equation}

Set $F:= K_0(b_i |i\in\N)$. Then $(F|K_0,v)$ is a unibranched extension with $vF=vK_0$ and $Fv=K_0v\left(\eta^{\frac{1}{p^i}} \, | \, i\in\N\right)$. Take further $b$ to be a root of the polynommial $X^p-b_0$. 
Note that $vb=\frac{1}{p}vd_0\geq -\frac{vp}{p}$. Since $b^p=a_0d_0$, we obtain that $(\frac{b}{d_1})^p=a_0$ and thus $\frac{b}{d_1}v=\eta^{1/p}\notin K_0v$. Hence, by the fundamental inequality we obtain that $(K_0(b)|K_0,v)$ is a unibranched extension with $vK_0(b)=vK_0$ and $K_0(b)v=K_0v(\eta^{1/p})$ of degree $p$ over $K_0v$. 

For every $k\geq 1$ set 
\[
c_k= b-\sum_{i=1}^k b_i\in K_0(b_k,b).
\]
Since all summands on the right hand side have value at least $-\frac{vp}{p}$, also \mbox{$vc_k\geq -\frac{vp}{p}$.} Hence Lemma~\ref{congruece} together with the fact that $b_i^p=b_i-b_{i-1}$ and $b_1^p=b_1+b_0$ yields that
\[
0=b^p-b_0= \left( c_k+\sum_{i=1}^k b_i  \right)^p -b_0\equiv c_k^p+\sum_{i=1}^k b_i^p-b_0\equiv 
c_k^p+b_k \, (\mod\mathcal{O}_v).
\]
Hence $v(c_k^p+b_k)\geq 0$ and thus 
\[
v\left(\left(\frac{c_k}{d_{k+1}}\right)^p+a_k\right)
=v\left(\frac{c_k^p}{d_k}+\frac{b_k}{d_k}\right)\geq -vd_k>0.
\]
Consequently, $\frac{c_k}{d_{k+1}}v=-\eta^{1/p^{k+1}}\notin K_0(b_k)v$. Similarly to the previous examples we show that the extension $(K_0(b_k,b)|K_0(b_k),v)$ is unibranched of degree $p$ with 
$vK_0(b_k,b)=vK_0(b_k)$ and
$K_0(b_k,b)v=K_0v\left(\eta^{1/p^{k+1}}\right)$ of degree $p$ over $K_0(b_k)v$.
Therefore
$(F(b)|F,v)$ is a unibranched extension of degree $p$. Since $\zeta_p\in F$, it is a Galois extension. 
Moreover, $vF(b)=vF$ and  $F(b)v= Fv$. We thus obtain that $(F(b)|F,v)$ is a Kummer defect extension of degree $p$.
\EXend
\end{example}

Take a valued field $(K,v)$ of characteristic $0$ and  residue characteristic $p>0$ with  roughly  $p$-divisible value  group  and imperfect residue field $Kv$. Note that in the above construction we allow the case $K_0=K$, i.e., $K$ contains a sequence defined in~(\ref{di}). This holds in particular if  $(K,v)$ has a cross-section, that is,
a group homomorphism $s\colon vK\to K^{\times}$ such that $v\circ s=\textup{id}_{vK}$. 
As every $\aleph_1$-saturated valued field has a cross-section  (cf. Lemma 3.3.39 of~\cite{[ADH]}), the above construction is   possible in particular for $\aleph_1$-saturated valued fields. 

We present now another construction, which does not require the existence of the sequence~(\ref{di}). However, as in the proof of Theorem~\ref{charact_equal}, we need to pass first to the absolute ramification field. We will need the following observation. 
\begin{lemma}   \label{2ext}
Assume that $(K,v)$ is a valued field of characteristic 0 and residue characteristic $p>0$  with roughly $p$-divisible value group and imperfect residue field. Take $a\in\mathcal{O}_v$ with $av\notin (Kv)^p$. Then there are two linearly disjoint unibranched extensions $(E|K,v)$ and $(E'|K,v)$ of degree $p$ with $Ev=E'v=Kv(av^{1/p})$.
\end{lemma}
\begin{proof}
Since the value group of $(K,v)$ is roughly $p$-divisible, there is an element $d\in K$ such that $-\frac{vp}{p^2}<vd<0$. Take $b\in \tilde{K}$ to be a root of the polynomial 
\begin{equation} \label{b}
f(x)=X^p-ad^{p^2}
\end{equation}
and $c\in \tilde{K}$ to be a root of the polynomial 
\begin{equation} \label{c}
f(x)=X^p-X-ad^{p^2}.
\end{equation}
Set $E=K(b)$ and $E'=K(c)$. Since $v\frac{b^p}{d^{p^2}}=va=0$ we deduce that $\frac{b}{d^p}=(av)^{\frac{1}{p}}\notin Kv$ nad $vb=vd^p-\frac{vp}{p}$. Hence by the fundamental inequality $(E|K,v)$ is a unibranched extension of degree $p$ with $Ev=Kv((av)^{\frac{1}{p}}).$

Similarly, as $v(c^p-c)=vad^{p^2}<0$, we obtain that $vc<0$ and thus $vd^{p^2}=v(c^p-c)=vc^p<vc$. Consequently, $vc=vd^p>-\frac{vp}{p}$
\[
v\left( \left( \frac{c}{d^p}\right) ^p-a \right )=vc-vd^{p^2}>vc^p-vd^{p^2}=0
\]
and hence $\frac{c}{d^p}v=(av)^{\frac{1}{p}}$. Thus also $(E'|K,v)$ is a unibranched extension of degree $p$ with $E'v=Kv((av)^{\frac{1}{p}}).$

Set $e=b-c\in E'.E$. As $vb,vc>-\frac{vp}{p}$, we obtain that $ve\geq -\frac{vp}{p}$, and by Lemma~\ref{congruece} we obtain that
\[
0=b^p-ad^{p^2}=(e+c)^p-ad^{p^2}\equiv e^p+c^p-ad^{p^2}\equiv e^p+c\, (\modd \mathcal{O}_v).
\]
Hence in particular $v(e^p-c)\geq 0$ and this yields that
\[v\left( \left( \frac{e}{d}\right) ^p- \frac{c}{d^p} \right )\geq -vd^p>0.\]
This implies that $\frac{e}{d}v=\left(\frac{c}{d^p}v\right)^{\frac{1}{p}}= av^{\frac{1}{p^2}}$. Since $Ev=Kv((av)^{\frac{1}{p}})$, by the fundamental inequality we obtain in particular that $(E'.E|E,v)$ is an extension of degree $p$. This shows that $E$ and $E'$ are $K$-linearly disjoint. 
\end{proof}

\begin{example} \label{ex2_resf_mixed}
Take a  henselian valued field $(K,v)$ of characteristic 0 and residue characteristic $p>0$  with $p$-divisible value group $vK=(vK)_{vp}$  and imperfect residue field. Assume moreover that $(K,v)$ is a defectless field. Take $(K_0,v)$ to be the absolute ramification field of $(K,v)$. Note that $vK_0=(vK_0)_{vp}$ remains  $p$-divisible and $K_0v$ remains imperfect. Moreover, by Proposition~\ref{defecless_ramfield} the field $(K_0,v)$ is defectless. 

Take $a\in K_0$ with $\eta_0:=(av)^{\frac{1}{p}}\notin K_0v$. By Lemma~\ref{2ext} there are linearly disjoint  extensions $(L|K,v)$ and $(K_1|K,v)$ of degree $p$ such that  $K_1v=Lv=K_0v(\eta_0)$. 

Assume that we have constructed a chain  $K_0\subset K_1\subset \ldots\subset K_i$ of algebraic extensions such that each $(K_i|K_{i-1},v)$ is a Galois extension of degree $p$ and $L$ is linearly disjoint from $K_i$. Assume moreover that $L.K_{i-1}v=K_iv$. Since $L$ and $K_i$ are $K_0$-linearly disjoint, the extension $L.K_i|K_i$ is of degree $p$. Since $(K_0,v)$ is a defecless field, the extension $(L.K_i|K_0,v)$ is defectless. Thus also $(L.K_i|K_i,v)$ is defectless. Since $vK_i$ is $p$-divisible and $K_iv$ is separably-algebraically closed, we deduce that $L.K_iv|K_iv$ is a purely inseparable extension of degree $p$. Take $\eta_i\in L.K_iv\setminus K_iv$  and $a_i\in K_i$ with $a_iv=\eta_i^p$. By Lemma~\ref{2ext} there are linearly disjoint  extensions $(E|K_i,v)$ and $(E'|K_i,v)$ of degree $p$ such that  $Ev=E'v=K_iv(\eta_i)$. Since $E$ and $E'$ are $K_i$-linearly disjoint, at least one of them is different from $L.K_i$. If $E\neq L.K_i$, set $K_{i+1}=E$. Otherwise set $K_{i+1}=E'$. As $K_{i+1}\neq L.K_i$ and $K_0=K_0^r$ by the last assertion of Theorem~\ref{tame_valgp_res} we deduce that $K_{i+1}|K_i$ and $L.K_i|K_i$ are linearly disjoint Galois extensions of degree $p$. As also $L$ and $K_i$ are $K_0$-linearly disjoint, we obtain that $L$ and $K_{i+1}$ are $K_0$-linearly disjoint. Moreover, by the choice of $\eta_i$, we have that
\[
L.K_iv=K_iv(\eta_i)=K_{i+1}v.
\]

Set $F:=\displaystyle\bigcup_{i=1}^{\infty}K_i$. Since $L$ is $K_0$-linearly disjoint from each $K_i$, it is also $K_0$-linearly disjoint from $F$. Thus $L.F|F$ is a Galois extension of degree $p$. As the value group of $K_0$ is $p$-divisible, the same holds for $vF$. Thus $vL=vF$. Moreover,
\[
Fv=\bigcup_{i=1}^{\infty}K_iv=\bigcup_{i=1}^{\infty}L.K_iv=L.Fv.
\]
Hence $(L.F|F,v)$ is a nontrivial defect Galois extension of degree $p$. Hence $(F,v)$ is an algebraic extension of $(K,v)$ which admits a Galois defect extension of degree $p$. 
\EXend
\end{example}

The above construction leads to the following conclusion. 

\begin{lemma} \label{prop_resf_mixed}
Take a  henselian and defectless valued field $(K,v)$ of characteristic 0 and residue characteristic $p>0$  with  $p$-divisible value group $vK=(vK)_{vp}$ and imperfect residue field.  Then there is an algebraic extension $(F|K,v)$ such that the field $(F,v)$ is not defectless.

\end{lemma}
We are now able to state the relation between roughly tame fields and defectless fields. 

\begin{theorem}  \label{thm_rtame}
Assume that $(K,v)$ is a henselian field.  Then $(K,v)$ is roughly tame if and only if every algebraic extension of $(K,v)$ is a defectless field. 
\end{theorem}

\begin{proof}
Note first that by  Theorem~\ref{charact_equal} it is enough to show the equivalence in the mixed characteristic case. Assume thus that char$K=0$ and char$Kv=p>0$. Take the composition $v=v_0\circ v_1$, where  $v_0$ is the finest coarsening of $v$ that has residue characteristic 0. 

Since $(K,v)$ is henselian, $v$ extends  in a unique way to $\tilde{K}$. Denote this extension again by $v$.
Assume that the field $(K,v)$ is roughly tame. Then the valued field $(Kv_0,v_1)$ is tame.  Proposition~\ref{tame_impl_def} yields that every algebraic extension of $(Kv_0,v_1)$ is a defectless field. By Proposition~\ref{coars_mixed_prop}, every algebraic extension of $(K,v)$ is a defectless field. 

Assume now that $(K,v)$ is not roughly tame. We show that then there is an algebraic extension $(F|K,v)$ such that $(F,v)$ is not a defectless field. If (RTF3) does not hold, we can take $F=K$. If (RTF1) does not hold, then the existence of $(F,v)$ follows from Lemma~\ref{prop_valgp_mixed}.  

Assume now that $vK$ is roughly $p$-divisible, $(K,v)$ is a defectless field and $Kv$ is imperfect.   Then by Lemma~\ref{coars_hens} and Proposition~\ref{coars_mixed_prop} the core field $(Kv_0,v_1)$ of $(K,v)$ is henselian and defectless. Furthermore, 
$v_1(Kv_0)= (vK)_{vp}$ is $p$-divisible and $(Kv_0)v_1=Kv$ is imperfect.
By Lemma~\ref{prop_resf_mixed}, the core field of $(K,v)$ admits an algebraic extension which is not a defectless field. Hence Proposition~\ref{coars_mixed_prop} yields that also $(K,v)$ admits an algebraic extension which is not a defectless field.
\end{proof}

The following fact is an immediate consequence of  Theorem~\ref{thm_rtame} and Lemma~\ref{rtame_equchar}.
\begin{corollary}
Assume that $(K,v)$ is a henselian field of characteristic $0$ and positive residue characteristic $p$ such that $(vK)_{vp}=vK$. Then $(K,v)$ is tame if and only if every algebraic extension of $(K,v)$ is a defectless field. 
\end{corollary}
Note that if the value group of $(K,v)$ is of rank 1, then $(vK)_{vp}=vK$. Together with the above corollary this proves Theorem~\ref{charact_equal} in the mixed characteristic case.

\sn

{\bf Proof of Theorem~\ref{thm_henrtame}.}\\
Suppose $K^h$ is not a roughly tame field. From  Theorem~\ref{thm_rtame}  we know that there exists an algebraic extension $L|K^h$ such that $L$ is not a defectless field. Of course $L|K$ is an algebraic extension, therefore there is an algebraic extension of K that is not a defectless field.

To prove the converse assume that $K^h$ is roughly tame and take an algebraic extension $L|K.$
We have that $L^h$ is an algebraic extension of a roughly tame field. Thus from Lemma \ref{rt} also $L^h$ is a roughly tame field and therefore it is defectless. From Theorem \ref{iffdef} we have that the field $L$ is defectless.
\QED

\section{Roughly tame, semitame and rdr fields} ~\label{sec_comar}

Studies of the structure of defect extensions of valued fields and possible extensions of the class of tame fields which preserve some of their crucial properties, led us to introduce other generalizations of the notion of tame fields. We recall now two of such generalizations introduced in~\cite{[KR]} and describe their relations with roughly tame fields.

Take a nontrivially valued field $(K,v)$ of residue characteristic exponent $p$ and denote by $\mathcal{O}_{\hat K}$ the valuation ring of the completion of $(K,v)$. Then $(K,v)$ is called a \bfind{roughly deeply ramified field} (in short an \bfind{rdr field}), if it satisfies:
\sn
{\bf (1)} the homomorphism
\[
\mathcal{O}_{\hat K}/p\mathcal{O}_{\hat K} \ni x\mapsto x^p\in \mathcal{O}_{\hat K}/p\mathcal{O}_{\hat K}
\]
is surjective;
\sn
{\bf (2)} if $\chara Kv>0$, then $vp$ is not the smallest positive element 
in $vK$.
\sn
We call $(K,v)$ a \bfind{semitame field} if it satisfies condition \bfind{ (1)} and
\sn
{\bf (2')} $vK$ is $p$-divisible.
\sn

Note that every nontrivially valued field of residue characteristic 0 is an rdr and a semitame field. Moreover, we have the following relations (cf. Theorem~1.2 of~\cite{[KR]})
\begin{theorem} \label{semitame_gdr}
Take a nontrivially valued field $(K,v)$. 
\begin{enumerate}
\item If the field $(K,v)$ is tame, then it is semitame.
\item If $(K,v)$ is semitame, then it is an rdr field. 
\item If $(K,v)$ is of equal characteristic or $vK$ is of rank 1, then the notions of semitame and rdr fields are equivalent. 
\end{enumerate}
\end{theorem}

Take a valued field  $(K,v)$. If it is of mixed characteristic, then we set
$K':=K(\zeta_p)$, where $\zeta_p$ is a primitive $p$-th root of unity. Otherwise we set $K'=K$. We call $(K,v)$ an \bfind{independent defect field} if for some extension of $v$ to 
$\tilde{K}$, all Galois defect extensions of $(K',v)$ of degree $p$ have independent
defect in the sense of ~\cite{[KR]}. Note that every defectless field is trivially an independent defect field. 

Assume that $(K,v)$ is a valued field of residue characteristic $p>0$. If $\chara K=p$  set 
$\crf (K,v):= Kv$.  If $\chara K=0$, consider the decomposition $v=v_0\circ v_p\circ \bar{v}$, where $v_0$ is the finest coarsening of $v$ 
that has residue characteristic 0, $v_p$ is a rank 1 valuation on $Kv_0\,$, and the valuation
$\bar{v}$ is  induced by $v$ on the residue field of $v_p$. Note that $v_p\circ \bar{v}$ is the core valuation. In this case set  $\crf (K,v):= Kv_0v_p$. 
To describe the relation between the class of roughly tame fields and rdr fields we will need a characterization of rdr fields proved in Theorem 1.10 of~\cite{[KR]}. 
\begin{theorem}  \label{gdr_charact}
Take a nontrivially valued field $(K,v)$ of positive residue characteristic $p$.  Then $(K,v)$  is an rdr field 
if and only if the following conditions hold:
\begin{enumerate}[(1)]
\item $vK$ is roughly $p$-divisible,
\item $\crf (K,v)$ is perfect,
\item $(K,v)$ is an independent defect field.
\end{enumerate}
\end{theorem}

This allows us to prove the following result.

\begin{theorem}
Take a nontrivially valued henselian field $(K,v)$.
\begin{enumerate}
\item If the field $(K,v)$ is semitame and  roughly tame, then it is tame. 
\item If $(K,v)$ is a roughly tame field, then it is an rdr field. 
\end{enumerate}
\end{theorem}
\begin{proof}
Assume that $(K,v)$ is semitame. Then (TF1) holds. If $(K,v)$ is in addition roughly tame, then by Proposition~\ref{roughly_charact} also conditions (TF2) and (TF2) hold. Hence $(K,v)$ is tame. 

Suppose now that $(K,v)$ is a roughly tame field. If $\chara Kv=0$, then the assertion is obvious. Thus we can assume that $\chara Kv=p>0$. 
 Then by Proposition~\ref{roughly_charact} the value group $vK$ is roughly $p$-divisible. Moreover, $(K,v)$ is a defectless field, hence it is in a trivial way an independent defect field. It remains to show that $\crf (K,v)$ is perfect. If $\chara K>0$, then $\crf (K,v)=Kv$ is perfect by condition (RTF2). Thus we can assume that $(K,v)$ is of mixed characteristic.  Since $\bar{v}(Kv_0v_p)$ is a subgroup of $(vK)_{vp}$, it is also $p$-divisible. Moreover, the residue field $Kv$ of $(Kv_0v_p,\bar{v})$ is perfect.  Hence any purely inseparable extension of $(Kv_0v_p, \bar{v})$ is a unibranched immediate extension. Suppose there were a nontrivial purely inseparable extension of $Kv_0v_p$. Then it would be a defect extension. On the other hand, the field $(K,v)$ is  defectless by condition (RTF3). Hence Proposition~\ref{coars_def} implies that also  $(Kv_0v_p, \bar{v})$ is defectless, a contradiction. Consequently, also in this case $\crf (K,v)=Kv_0v_p$ is perfect. Therefore, by Theorem~\ref{gdr_charact} $(K,v)$ is a rdr field.
\end{proof}

\newcommand{\lit}[1]{\bibitem{#1}}

\abaddress
\psaddress
\end{document}